\documentclass[twoside]{article}%
\usepackage{amssymb}
\usepackage{amsfonts}
\usepackage{amsmath}
\usepackage{graphicx}%
\providecommand{\U}[1]{\protect\rule{.1in}{.1in}}
\newtheorem{theorem}{Theorem}

\newtheorem{lemma}[theorem]{Lemma}

\newenvironment{proof}[1][Proof]{\noindent\textbf{#1.} }{\ \rule{0.5em}{0.5em}}
\begin{document}

\title{\textbf{Stabilities for Euler-Poisson Equations with Repulsive Forces in
}$R^{N}$}
\author{M\textsc{anwai Yuen}\thanks{E-mail address: nevetsyuen@hotmail.com }\\\textit{Department of Applied Mathematics, }\\\textit{The Hong Kong Polytechnic University,}\\\textit{Hung Hom, Kowloon, Hong Kong}}
\date{Revised 3-Jan-2010}
\maketitle

\begin{abstract}
This article extends the previous paper in "M.W. Yuen, \textit{Stabilities for
Euler-Poisson Equations in Some Special Dimensions}, J. Math. Anal. Appl.
\textbf{344} (2008), no. 1, 145--156.", from the Euler-Poisson equations for
attractive forces to the repulsive ones in $R^{N}$ $(N\geq2)$. The similar
stabilities of the system are studied. Additionally, we explain that it is
impossible to have the density collapsing solutions with compact support to
the system with repulsive forces for $\gamma>1$.


\end{abstract}

\section{Introduction}

The semi-conductor models can be formulated by the isentropic Euler-Poisson
equation with repulsive forces in the following form:
\begin{equation}
\left\{
\begin{array}
[c]{rl}%
{\normalsize \rho}_{t}{\normalsize +\nabla\cdot(\rho u)} & {\normalsize =}%
{\normalsize 0,}\\
{\normalsize (\rho u)}_{t}{\normalsize +\nabla\cdot(\rho u\otimes u)+\nabla
P+\beta\rho u} & {\normalsize =}{\normalsize \rho\nabla\Phi,}\\
{\normalsize \Delta\Phi(t,x)} & {\normalsize =\alpha(N)}{\normalsize \rho,}%
\end{array}
\right.  \label{Euler-Poisson}%
\end{equation}
where $\alpha(N)$ is a constant related to the unit ball in $R^{N}$:
$\alpha(1)=2$; $\alpha(2)=2\pi$; For $N\geq3,$%
\begin{equation}
\alpha(N)=N(N-2)V(N)=N(N-2)\frac{\pi^{N/2}}{\Gamma(N/2+1)},
\end{equation}
where $V(N)$ is the volume of the unit ball in $R^{N}$ and $\Gamma$ is the
Gamma function. As usual, $\rho=\rho(t,x)$ and $u=u(t,x)\in\mathbf{R}^{N}$ are
the density and the velocity respectively. $P=P(\rho)$\ is the pressure and
$\beta\geq0$ is the frictional damping constant. In the above system, the
self-repulsive potential field $\Phi=\Phi(t,x)$\ is determined by the density
$\rho$ through the Poisson equation. The equations (\ref{Euler-Poisson})$_{1}$
and (\ref{Euler-Poisson})$_{2}$ are the compressible Euler equations with
forcing term. The equation (\ref{Euler-Poisson})$_{3}$ is the Poisson equation
through which the potential with repulsive forces is determined by the density
distribution of the electrons themselves. Thus, we called the system
(\ref{Euler-Poisson}) the Euler-Poisson equations with repulsive forces. In
this case, the equations can be viewed as a semiconductor model. See \cite{C},
\cite{Lions} for a detail about the system. For some fixed $K\geq0$, we have a
$\gamma$-law on the pressure $P(\rho)$, i.e.%
\begin{equation}
{\normalsize P}\left(  \rho\right)  {\normalsize =K\rho}^{\gamma}\doteq
\frac{{\normalsize \rho}^{\gamma}}{\gamma},\label{gamma}%
\end{equation}
for $K>0$, which is a common hypothesis. When $K=0$, the pressureless system
can be used as models in plasma physics \cite{BB}. The constant $\gamma
=c_{P}/c_{v}\geq1$, where $c_{P}$ and $c_{v}$\ are the specific heats per unit
mass under constant pressure and constant volume respectively, is the ratio of
the specific heats. Additionally, the fluid is called isothermal if $\gamma=1$.

The Poisson equation (\ref{Euler-Poisson})$_{3}$ can be solved as%
\begin{equation}
{\normalsize \Phi(t,x)=}\int_{R^{N}}G(x-y)\rho(t,y){\normalsize dy,}%
\end{equation}
where $G$ is the Green's function for the Poisson equation in the
$N$-dimensional spaces defined by
\begin{equation}
G(x)\doteq\left\{
\begin{array}
[c]{ll}%
|x|, & N=1;\\
\log|x|, & N=2;\\
\frac{-1}{|x|^{N-2}}, & N\geq3.
\end{array}
\right.
\end{equation}
If we seek solutions in radial symmetry with the radius $r=\left(  \sum
_{i=1}^{N}x_{i}^{2}\right)  ^{1/2}$, the Poisson equation (\ref{Euler-Poisson}%
)$_{3}$ is transformed to%
\begin{equation}
{\normalsize r^{N-1}\Phi}_{rr}\left(  {\normalsize t,x}\right)  +\left(
N-1\right)  r^{N-2}\Phi_{r}{\normalsize =}\alpha\left(  N\right)
{\normalsize \rho r^{N-1},}%
\end{equation}%
\begin{equation}
\Phi_{r}=\frac{\alpha\left(  N\right)  }{r^{N-1}}\int_{0}^{r}\rho
(t,s)s^{N-1}ds.
\end{equation}
We can seek the radial symmetry solutions
\begin{equation}
\rho(t,\vec{x})=\rho(t,r)\text{ and }\vec{u}=\frac{\vec{x}}{r}V(t,r)=:\frac
{\vec{x}}{r}V.
\end{equation}
By standard computation, the Euler-Poisson equations in radial symmetry can be
written in the following form:%
\begin{equation}
\left\{
\begin{array}
[c]{c}%
\rho_{t}+V\rho_{r}+\rho V_{r}+\dfrac{N-1}{r}\rho V=0,\\
\rho\left(  V_{t}+VV_{r}\right)  +P_{r}(\rho)=\rho\Phi_{r}\left(  \rho\right)
.
\end{array}
\right.  \label{eq12345}%
\end{equation}
Perthame discovered the blowup results for $3$-dimensional pressureless system
with repulsive forces \cite{P}. In short, all the results above rely on the
solutions with radial symmetry:
\begin{equation}%
\begin{array}
[c]{rl}%
V_{t}+VV_{r} & {\normalsize =}\frac{\alpha(N)}{r^{N-1}}\int_{0}^{r}%
\rho(t,s)s^{N-1}ds.
\end{array}
\end{equation}
And the Emden ordinary differential equations were deduced on the boundary
point of the solutions with compact support:%
\begin{equation}
\frac{D^{2}R}{Dt^{2}}=\frac{M}{R^{N-1}},\text{ }R(0,R_{0})=R_{0}\geq0,\text{
}\dot{R}(0,R_{0})=0,
\end{equation}
where $\frac{dR}{dt}:=V$ and $M$ is the mass of the solutions, along the
characteristic curve. They showed the blowup results for the $C^{1}$ solutions
of the system (\ref{eq12345}).

Very recently, Yuen showed in \cite{Y3} that the classical non-trivial
solutions $\left(  \rho,V\right)  $ for the Euler or Euler-Poisson equations
with repulsive forces, in radial symmetry, with compact support in $\left[
0,R\right]  $, where $R$ is a positive constant (which is $V(t,0)=0;\rho
(t,r)=0,V(t,r)=0$ for $r\geq R$) and the initial velocity such that:
\begin{equation}
H_{0}=\int_{0}^{R}V_{0}dr>0,
\end{equation}
blow up on or before the finite time $T=2R/H_{0},$ for pressureless fluids
$(K=0)$ or $\gamma>1$. \newline The system with attractive forces were studied
in \cite{DLYY}, \cite{M}, \cite{M1}, \cite{MP}, \cite{Y1}.and \cite{Y2}. This
article extends the previous paper in \cite{Y2}, from the Euler-Poisson
equations for attractive forces with or without frictional damping to the
repulsive ones in $R^{N}$ $(N\geq2)$. The similar stabilities of the system
are studied.

In the last section, we exclude the possibility of collapsing solutions for
this system. The non-existence of collapsing solutions can be shown by the
simple argument for the energy function:%
\begin{equation}%
\begin{array}
[c]{cc}%
E(t)=\int_{\Omega}\left(  \frac{1}{2}\rho\left\vert u\right\vert ^{2}+\frac
{1}{\gamma-1}P-\frac{1}{2}\rho\Phi\right)  dx, & \text{ for }\gamma>1.
\end{array}
\end{equation}

\begin{theorem}
\label{kooko}For the classical solutions with compact support of the
Euler-Poisson equations with repulsive force, (\ref{Euler-Poisson}), in
$R^{N}$ $(N\geq2)$ with $\gamma>1$ or without pressure $(K=0)$, there is no
collapsing phenomenon where part of the density $\rho(t,x)$ collapses to a point.
\end{theorem}

\section{Stabilities}

In this section, we study the stabilities of the Euler-Poisson equations with
repulsive forces, (\ref{Euler-Poisson}), in $R^{N}(N\geq2)$. The total energy
can be defined by,%
\begin{equation}
\left\{
\begin{array}
[c]{c}%
\begin{array}
[c]{cc}%
E(t)=\int_{\Omega}\left(  \frac{1}{2}\rho\left\vert u\right\vert ^{2}+\frac
{1}{\gamma-1}P-\frac{1}{2}\rho\Phi\right)  dx, & \text{ for }\gamma>1,
\end{array}
\\%
\begin{array}
[c]{cc}%
E(t)=\int_{\Omega}\left(  \frac{1}{2}\rho\left\vert u\right\vert ^{2}-\frac
{1}{2}\rho\Phi\right)  dx, & \text{ for without pressure.}%
\end{array}
\end{array}
\right.  \label{energy1}%
\end{equation}
For the system, we have the energy estimate:

\begin{lemma}
\label{ChangeRateofEnergy}For the Euler-Poisson equations,
(\ref{Euler-Poisson}), suppose the solutions $(\rho,u)$ have compact support
in $\Omega$. We have,
\begin{equation}
\overset{\cdot}{E}(t)=-\beta\int_{\Omega}\rho\left\vert u\right\vert
^{2}dx\leq0.
\end{equation}

\end{lemma}

Initially, Sideris used the second inertia function
\begin{equation}
{\normalsize H(t)=}\int_{\Omega}{\normalsize \rho(t,x)}\left\vert
{\normalsize x}\right\vert ^{2}{\normalsize dx,} \label{eq21}%
\end{equation}
to study instability results for the Euler equations \cite{SI}. After that in
\cite{MP}, the instability result of the Euler-Poisson equations with
attractive forces, in radial symmetry, was obtained for $\gamma\geq4/3$ and
$N=3$. For the corresponding cases in $R^{N}$, with non-radial symmetry, were
studied \cite{DLYY}, \cite{Y2}. By the standard computation for energy method,
it is clear to have the following lemma:

\begin{lemma}
\label{lem:secondinter}Consider $(\rho,u)$ is a solution with compact support
in $\Omega$ \ for the Euler-Poisson equations, (\ref{Euler-Poisson}) with
$\beta=0$. We have%
\begin{equation}
\left\{
\begin{array}
[c]{cc}%
\overset{\cdot\cdot}{H}(t)=2\int_{\Omega}\left[  (\rho\left\vert u\right\vert
^{2}+NP)dx-\frac{N-2}{2}\rho\Phi\right]  dx, & \text{for }N\geq3;\\
\overset{\cdot\cdot}{H}(t)=2\int_{\Omega}(\rho\left\vert u\right\vert
^{2}+2P)dx+M^{2}, & \text{for }N=2.
\end{array}
\right.  \label{eq22}%
\end{equation}

\end{lemma}

By applying the above lemma, we could have:

\begin{theorem}
Suppose $(\rho,u)$ is a global classical solution in the Euler-Poisson
equations, (\ref{Euler-Poisson}) with $\gamma>1$, without frictional damping
$(\beta=0)$. We have\newline(1)for $N\geq3,$%
\begin{equation}
\underset{t\rightarrow\infty}{\lim}\inf\frac{R(t)}{t}\geq\left[
\frac{{\normalsize \inf(2,N(\gamma-1),N-2)E}}{M}\right]  ^{1/2};
\end{equation}
(2)for $N=2,$\newline%
\begin{equation}
\underset{t\rightarrow\infty}{\lim}\inf\frac{R(t)}{t}\geq\sqrt{\frac{1}{2M}};
\end{equation}
(3)for $N\geq2,$
\begin{equation}
\underset{t\rightarrow\infty}{\lim}\inf\frac{R(t)}{t}{\normalsize \geq}\left[
\frac{NKM^{\gamma-1}}{\left\vert \Omega\right\vert ^{(\gamma-1)}}\right]
^{1/2},
\end{equation}
with $R(t)=\max_{x\in\Omega(t)}\left\{  \left\vert x\right\vert \right\}  $.
Here%
\begin{equation}
M=\int_{\Omega}\rho(t,x)dx,
\end{equation}
is the total mass which is constant for any classical solution and $\left\vert
\Omega\right\vert $ is the fixed volume of $\Omega$.
\end{theorem}

\begin{proof}
(1)For $N\geq3,$ we have the positive energy function $E\geq0$. We can get
from Lemma \ref{lem:secondinter},
\begin{equation}
\overset{\cdot\cdot}{H}(t)=2\left\{  \int_{\Omega}\left[  \rho\left\vert
u\right\vert ^{2}+NP\right]  dx-\frac{N-2}{2}\int_{\Omega}\rho\Phi dx\right\}
\geq2\inf(2,N(\gamma-1),N-2)E.
\end{equation}
That is%
\begin{equation}
{\normalsize H(t)\geq H(0)+}\overset{\cdot}{H}{\normalsize (0)t+\inf
(2,N(\gamma-1),N-2)Et}^{2}.
\end{equation}
On the other hand, we obtain,%
\begin{equation}
{\normalsize H(0)+}\overset{\cdot}{H}{\normalsize (0)t+\inf(2,N(\gamma
-1),N-2)Et}^{2}{\normalsize \leq H}\left(  t\right)  {\normalsize \leq
R(t)}^{2}{\normalsize M.}%
\end{equation}
That is,%
\begin{equation}
O(\frac{1}{t})+{\normalsize \inf(2,N(\gamma-1),N-2)E\leq}\frac{R(t)^{2}%
M}{t^{2}},
\end{equation}%
\begin{equation}
\underset{t\rightarrow\infty}{\lim}\inf\frac{R(t)}{t}\geq\left[
\frac{{\normalsize \inf(2,N(\gamma-1),N-2)E}}{M}\right]  ^{1/2}.
\end{equation}
For $N=2,$ we have%
\begin{equation}
\overset{\cdot\cdot}{H}(t)=2\int_{\Omega}(\rho\left\vert u\right\vert
^{2}+2P)dx+M^{2}\geq M^{2},
\end{equation}%
\begin{equation}
\frac{M^{2}}{2}t^{2}+C_{0}t+C_{1}\leq H(t)\leq{\normalsize R(t)}%
^{2}{\normalsize M,}%
\end{equation}%
\begin{equation}
\underset{t\rightarrow\infty}{\lim}\inf\frac{R(t)}{t}\geq\sqrt{\frac{1}{2M}}.
\end{equation}
For $N\geq2$, we obtain
\begin{equation}
{\normalsize M=}\int_{\Omega}{\normalsize \rho dx\leq}\left(  \int_{\Omega
}\rho^{\gamma}dx\right)  ^{1/\gamma}\left\vert \Omega\right\vert
^{(\gamma-1)/\gamma}, \label{pressureinequaltiy}%
\end{equation}
and
\begin{equation}
\overset{\cdot\cdot}{H}(t)=2\left\{  \int_{\Omega}\left[  \rho\left\vert
u\right\vert ^{2}+NP\right]  dx-\frac{N-2}{2}\int_{\Omega}\rho\Phi dx\right\}
\geq2\int_{\Omega}NPdx.
\end{equation}
From the inequality (\ref{pressureinequaltiy}), it is clear to have%
\begin{equation}
\overset{\cdot\cdot}{H}(t){\normalsize \geq}2NK\left\vert \Omega\right\vert
^{1-\gamma}{\normalsize M}^{\gamma}{\normalsize >0,}%
\end{equation}%
\begin{equation}
{\normalsize H(0)+}\overset{\cdot}{H}(0){\normalsize t+NK}\left\vert
\Omega\right\vert ^{1-\gamma}{\normalsize M}^{\gamma}{\normalsize t}%
^{2}{\normalsize \leq H(t)\leq R(t)}^{2}{\normalsize M,}%
\end{equation}%
\begin{equation}
{\normalsize O(}\frac{1}{t}{\normalsize )+}NK\left\vert \Omega\right\vert
^{1-\gamma}{\normalsize M}^{\gamma}{\normalsize \leq}\frac{R(t)^{2}M}{t^{2}}.
\end{equation}
This gives%
\begin{equation}
\underset{t\rightarrow\infty}{\lim}\inf\frac{R(t)}{t}{\normalsize \geq}\left[
\frac{NKM^{\gamma-1}}{\left\vert \Omega\right\vert ^{(\gamma-1)}}\right]
^{1/2}.
\end{equation}
The proof is completed.
\end{proof}

By applying the similar method, in the above section to the system,
(\ref{Euler-Poisson}), with frictional damping constant $(\beta>0)$, the below
theorem is followed:

\begin{theorem}
Suppose $(\rho,u)$ is a global classical solution with compact support in the
system, (\ref{Euler-Poisson}), with frictional damping $(\beta>0)$. We have
for $N\geq2,$%
\begin{equation}
\underset{t\rightarrow\infty}{\lim}\inf\frac{R(t)}{t^{1/2}}{\normalsize \geq
}\left(  \frac{2\beta NKM^{\gamma-1}}{\left\vert \Omega\right\vert
^{(\gamma-1)}}\right)  ^{1/2},
\end{equation}
with $R(t)=\max_{x\in\Omega(t)}\left\{  \left\vert x\right\vert \right\}  $.
\end{theorem}

\begin{proof}
For $N\geq2$, we get,%
\begin{equation}
\overset{\cdot}{H}(t)=\int_{\Omega}2x\rho udx,
\end{equation}
and%
\begin{equation}
\overset{\cdot\cdot}{H}(t)=2\int_{\Omega}x\left[  -\nabla\cdot\left(  \rho
u\otimes u\right)  -\nabla P+\rho\nabla\Phi-\beta\rho u\right]  dx.
\end{equation}
Additionally, it can be arranged as the following:
\begin{align}
\overset{\cdot\cdot}{H}(t)  &  =2\int_{\Omega}\left[  x-\nabla\cdot\left(
\rho u\otimes u\right)  -\nabla P+\rho\nabla\Phi\right]  dx-\beta
\overset{\cdot}{H}(t),\\
\overset{\cdot\cdot}{H}(t)+\frac{1}{\beta}\overset{\cdot}{H}(t)  &  =2\left\{
\int_{\Omega}\left[  \rho\left\vert u\right\vert ^{2}+NP\right]  dx-\frac
{N-2}{2}\int_{\Omega}\rho\Phi dx\right\}  \geq2\int NPdx\geq2NK\left\vert
\Omega\right\vert ^{1-\gamma}{\normalsize M}^{\gamma}>0.
\end{align}
Therefore, we are able to obtain the inequality,%
\begin{equation}
C_{3}{\normalsize +}C_{4}e^{-\beta t}{\normalsize +}2\beta NK\left\vert
\Omega\right\vert ^{1-\gamma}{\normalsize M}^{\gamma}t{\normalsize \leq
H(t)\leq R(t)}^{2}{\normalsize M,} \label{ff1}%
\end{equation}%
\begin{equation}
{\normalsize O(}\frac{1}{t}{\normalsize )+2\beta}NK\left\vert \Omega
\right\vert ^{1-\gamma}{\normalsize M}^{\gamma}{\normalsize \leq}%
\frac{R(t)^{2}M}{t}.
\end{equation}
This gives%
\begin{equation}
\underset{t\rightarrow\infty}{\lim}\inf\frac{R(t)}{t^{1/2}}{\normalsize \geq
}\left(  \frac{2\beta NKM^{\gamma-1}}{\left\vert \Omega\right\vert
^{(\gamma-1)}}\right)  ^{1/2}.
\end{equation}
The proof is completed.
\end{proof}

\section{Non-existence of Collapsing Solution}

In this section, we explain the idea that there is no possibility to have a
density collapsing solution, with compact support for the Euler-Poisson
equations with repulsive forces:

We restate their energy for $\gamma>1$ or the pressureless fluids is:
\begin{equation}
E(t)=\int_{\Omega}\left(  \frac{1}{2}\rho\left\vert u\right\vert ^{2}+\frac
{1}{\gamma-1}P-\frac{1}{2}\rho\Phi\right)  dx\geq-\int_{\Omega}\frac{1}{2}%
\rho\Phi dx.
\end{equation}
When a $\delta$-shock exists for the density function $\rho(t,x)$, the
potential energy function, with $N\geq3$, becomes to be infinite:%
\begin{equation}
-\int_{\Omega}\rho\Phi dx=\int_{\Omega}\rho(t,x)\int_{\Omega}\frac{\rho
(t,y)}{\left\vert x-y\right\vert ^{N-2}}dydx=\underset{\epsilon\rightarrow
0^{+}}{\lim}\int_{\Omega}\delta(t,x)\int_{\Omega}\frac{\delta(t,y)}%
{\epsilon^{N-2}}dydx=\infty\text{.}%
\end{equation}
With $N=2$, the situation is similar:%
\begin{equation}
-\int_{\Omega}\rho\Phi dx=-\int_{\Omega}\delta(t,x)\int_{\Omega}\delta
(t,y)\ln\left\vert x-y\right\vert dydx=\infty.
\end{equation}
Therefore, the total energy of the $\delta$-shock density solutions must be
infinite. However, for the classical solutions with compact support, the
initial energy is finite. By the energy estimate in Lemma
\ref{ChangeRateofEnergy}, we have,%
\begin{equation}
E(t)\leq E(0).
\end{equation}
If the total energy is finite, it is impossible to obtain the density
collapsing solutions. It is clear to have Theorem \ref{kooko}.


\begin{thebibliography}{99}                                                                                               %


\bibitem {BB}F. Bouchut and G. Bonnaud, \textit{Numerical Simulation of
Relativistic Plasmas in Hydrodynamic Regime}, ZAMM \textbf{76} (1996), 287-290.

\bibitem {C}C. F. Chen, \textit{Introduction to Plasma Physics and Controlled
Fusion}, Plenum, New York (1984)

\bibitem {DLYY}Y.B. Deng, T.P. Liu, T. Yang and Z.A. Yao, \textit{Solutions of
Euler-Poisson Equations for Gaseous Stars}, Arch. Rational Mech. Anal.
\textbf{164 }(2002), 261-285.

\bibitem {G}R. T. Glassey, \textit{The Cauchy Problem in Kinetic Theory}.
Society for Industrial and Applied Mathematics (SIAM), Philadelphia, PA, 1996.

\bibitem {Lions}P.L. Lions, \textit{Mathematical Topics in Fluid Mechanics}.
Vols. 1, 2, 1998, Oxford: Clarendon Press, 1998.

\bibitem {M}T. Makino, \textit{On a Local Existence Theorem for the Evolution
Equation of Gaseous Stars}, Patterns and waves, 459--479, Stud. Math. Appl.,
18, North-Holland, Amsterdam, 1986.

\bibitem {M1}T. Makino, \textit{Blowing up Solutions of the Euler-Poission
Equation for the Evolution of the Gaseous Stars}, Transport Theory and
Statistical Physics \textbf{21} (1992), 615-624.

\bibitem {MP}T. Makino and B. Perthame, \textit{Sur les Solutions a symmetric
spherique de lequation d'Euler-poisson Pour levolution d'etoiles
gazeuses,(French) [On Radially Symmetric Solutions of the Euler-Poisson
Equation for the Evolution of Gaseous Stars], }Japan J. Appl. Math. \textbf{7}
(1990), 165-170.

\bibitem {P}B. Perthame, \textit{Nonexistence of Global Solutions to
Euler-Poisson Equations for Repulsive Forces}, Japan J. Appl. Math. \textbf{7}
(1990), 363--367.

\bibitem {SI}T.C. Sideris, \textit{Formation of Singularities in
Three-dimensional Compressible Fluids}, Comm. Math. Phys. \textbf{101} (1985),
No. 4, 475-485.

\bibitem {Y}M.W. Yuen, \textit{Blowup Solutions for a Class of Fluid Dynamical
Equations in }$R^{N}$, J. Math. Anal. Appl. \textbf{329} (2)(2007), 1064-1079.

\bibitem {Y1}M.W. Yuen, \textit{Analytical Blowup Solutions to the
2-dimensional Isothermal Euler-Poisson Equations of Gaseous Stars}, J. Math.
Anal. Appl. \textbf{341 (}1\textbf{)(}2008\textbf{), }445-456.

\bibitem {Y2}M.W. Yuen, \textit{Stabilities for Euler-Poisson Equations in
Some Special Dimensions}, J. Math. Anal. Appl. \textbf{344} (2008), no. 1, 145--156.

\bibitem {Y3}M.W. Yuen, \textit{Blowup for the Euler and Euler-Poisson
Equations with Repulsive Forces}, pre-print, arXiv:1001.0380v1. 
\end{thebibliography}
\end{document}